\documentclass[english]{jams-l}

\newtheorem{theorem}{Theorem}[section]
\newtheorem{lemma}[theorem]{Lemma}

\newtheorem{corollary}[theorem]{Corollary}

\theoremstyle{definition}
\newtheorem{definition}[theorem]{Definition}

\numberwithin{equation}{section}

\theoremstyle{remark}
\newtheorem*{remark}{Remark}

\usepackage{graphicx}
\usepackage{a4wide}

\usepackage[fixlanguage]{babelbib}
\selectbiblanguage{english}

\usepackage[shortlabels]{enumitem}

\usepackage{amsfonts,amssymb,amsbsy,amsmath,mathrsfs,amsthm}

\usepackage{commath}

\usepackage{hyperref}
\hypersetup{pdfpagemode=UseNone, pdfstartview=FitH,
  colorlinks=true,urlcolor=red,linkcolor=blue,citecolor=black,
  pdftitle={Default Title, Modify},pdfauthor={Your Name},
  pdfkeywords={Modify keywords}}

\RequirePackage{babel}

\providecommand{\norm}[1]{ \lVert#1  \rVert}

\newcommand{\dx}{\, d x}
\newcommand{\dy}{\, d y}

\newcommand{\dla}{\, d \lambda}
\newcommand{\lp}{\left (}
\newcommand{\rp}{\right )}

\def\Xint#1{\mathchoice
   {\XXint\displaystyle\textstyle{#1}}%
   {\XXint\textstyle\scriptstyle{#1}}%
   {\XXint\scriptstyle\scriptscriptstyle{#1}}%
   {\XXint\scriptscriptstyle\scriptscriptstyle{#1}}%
   \!\int}
\def\XXint#1#2#3{{\setbox0=\hbox{$#1{#2#3}{\int}$}
     \vcenter{\hbox{$#2#3$}}\kern-.5\wd0}}

\def\dashint{\Xint-}

\usepackage{cite}
\makeatletter
\newcommand{\citecomment}[2][]{\citen{#2}#1\citevar}
\newcommand{\citeone}[1]{\citecomment{#1}}
\newcommand{\citetwo}[2][]{\citecomment[,~#1]{#2}}
\newcommand{\citevar}{\@ifnextchar\bgroup{;~\citeone}{\@ifnextchar[{;~\citetwo}{]}}}
\newcommand{\citefirst}{\@ifnextchar\bgroup{\citeone}{\@ifnextchar[{\citetwo}{]}}}
\newcommand{\cites}{[\citefirst}
\makeatother

\begin{document}

\title{Dyadic John--Nirenberg space}

\author{Juha Kinnunen}
\address{Department of Mathematics, Aalto University, P.O. Box 11100, FI-00076 Aalto, Finland}
\email{juha.k.kinnunen@aalto.fi}
\thanks{The research was supported by the Academy of Finland.}

\author{Kim Myyryl\"ainen}
\address{Department of Mathematics, Aalto University, P.O. Box 11100, FI-00076 Aalto, Finland}
\email{kim.myyrylainen@aalto.fi}



\begin{abstract}
We discuss the dyadic John--Nirenberg space that is a generalization of functions of bounded mean oscillation.
A John--Nirenberg inequality, which gives a weak type estimate for the oscillation of a function, is discussed in the setting of medians instead of integral averages.
We show that the dyadic maximal operator is bounded on the dyadic John--Nirenberg space and provide a method to construct nontrivial functions in the dyadic John--Nirenberg space.
Moreover, we prove that the John--Nirenberg space is complete. Several open problems are also discussed.
\end{abstract}

\maketitle

\section{Introduction}

The space of functions of bounded mean oscillation ($BMO$) was introduced by John and Nirenberg in~\cite{john_original}. 
Let $Q_0$ be a cube with sides parallel to the coordinate axis in $\mathbb R^n$. 
A function $f \in L^1(Q_0)$ belongs to $BMO(Q_0)$ if
\begin{equation}\label{bmo}
\sup \dashint_{Q} |f-f_{Q}| \dx < \infty,
\end{equation}
where the supremum is taken over all subcubes of $Q_0$. 
Throughout, we denote the integral average over a cube by a barred integral sign or $f_Q$.
A more general $BMO$-type space was also discussed in~\cite{john_original}.
A function $f \in L^1(Q_0)$ belongs to the John--Nirenberg space $JN_p(Q_0)$, $1<p<\infty$, if
\begin{equation}
\label{original_JNp}
\sup \sum_{i=1}^\infty |Q_i| \biggl( \dashint_{Q_i} |f-f_{Q_i}| \dx \biggr)^p < \infty,
\end{equation}
where the supremum is taken over countable collections  $\{Q_i\}_{i\in\mathbb N}$ of pairwise disjoint subcubes of $Q_0$.
The space $BMO(Q_0)$ is obtained as the limit of $JN_p(Q_0)$ as $p\to\infty$.
John~\cite{johnmedian} considered a way to define $BMO(Q_0)$ for any measurable function $f$ on $Q_0$ and
this approach has been developed further by Str\"omberg~\cite{stromberg} and Jawerth and Torchinsky~\cite{jawerth_torchinsky}.
In this case,~\eqref{bmo} is replaced with 
\begin{equation}\label{bmo_zero}
\sup\inf_{c\in\mathbb R}\inf\{a\ge0:|\{x\in Q:|f(x)-c|>a\}|<s|Q|\}<\infty,
\end{equation}
where the supremum is taken over all subcubes of $Q_0$ and $s$ is a fixed parameter with $0<s\le1$.
Perhaps the most common parameter value is $s=\frac12$ and, for $0<s \leq 1$, we obtain a biased notion of $s$-median.
Medians have been studied and applied in many problems; see for example~\cite{federer_herbert_william, fujii, gogatishvili, heikkinen, heikkinen_measuredensity, heikkinen_kinnunen, approxquasi, approxholder, jawerth_perez_welland, jawerth_torchinsky, johnmedian, karak, lerner, lerner_perez, medians, stromberg, weightedhardy, zhou}.

This paper discusses several new results related to the definition and properties of the John--Nirenberg space with $s$-medians (Definition~\ref{Def_medJN_eucli}).
In particular, this extends the median approach of $BMO$ in~\eqref{bmo_zero} to John--Nirenberg spaces. 
We restrict our attention to the dyadic case, that is, the cubes in~\eqref{original_JNp} are assumed to be dyadic subcubes of $Q_0$.
The dyadic structure has many advantages in the theory of John--Nirenberg spaces.
For some of our results, it does not matter whether we consider dyadic cubes or all subcubes of $Q_0$, but some results hold exclusively for dyadic cubes.
We study a John--Nirenberg inequality for the dyadic John--Nirenberg space with $s$-medians (Theorem~\ref{thm_JN_eucli}).
Our proof is based on relatively standard arguments.
Related questions on metric measure spaces have been studied by Lerner and P\'{e}rez~\cite{lerner_perez} and Myyryl\"ainen~\cite{myyrylainen}. 
We reconsider dyadic versions of these results in the Euclidean context.
As a consequence (Corollary~\ref{equivalence_euclidean}), we show that the dyadic John--Nirenberg space with medians coincides with the dyadic John--Nirenberg space with integral averages.
Thus, it does not matter which one we consider. However, assumptions in the median approach are initially weaker, since the function does not need to be integrable.

Bennett, DeVore and Sharpley~\cite{bennett} showed that the Hardy--Littlewood maximal operator is bounded on $BMO$.
For a short proof, we refer to Chiarenza and Frasca~\cite{chiarenza}.
We show that the dyadic maximal operator is bounded on the dyadic John--Nirenberg space $JN^d_p(Q_0)$ (Theorem~\ref{thm_mfbound}).
To our knowledge, this result is new.
The proof is based on the John--Nirenberg inequality.
A similar argument, with the weak type estimate for the maximal operator, gives an $L^1$ result for the dyadic maximal operator (Theorem~\ref{construct_dyaJNp}).
Using this result together with a theorem of Stein~\cite{stein_LlogL}, we obtain a method to construct functions in $JN^d_p(Q_0) \setminus L^p(Q_0)$.
This complements results by Dafni, Hyt\"onen, Korte and Yue~\cite{korte} in the dyadic case.
Motivated by Theorem~\ref{construct_dyaJNp}, it is an open question whether there exists a Coifman-Rochberg~\cite{coifmanrochberg} type characterization for the dyadic John--Nirenberg space.
A one-dimensional example in Section~4 demonstrates that the $L^1$ result in its generality does not hold for the standard John--Nirenberg space.
The standard $BMO$ is complete with respect to the $BMO$ seminorm; see~\cite{neri}.
We prove that the dyadic John--Nirenberg space is complete (Theorem~\ref{completeness}).
This also holds for the standard John--Nirenberg space $JN_p(Q_0)$.
The connection between the dyadic $BMO$ and the standard $BMO$ has been studied by Garnett and Jones in~\cite{garnett}.
The corresponding result is also true for the John--Nirenberg spaces.

\section{Preliminaries}

The Lebesgue measure of a measurable subset  $A$ of $\mathbb{R}^n$ is denoted by $|A|$.
The integral average of $f \in L^1(A)$ in $A$, with $0<|A|<\infty$, is denoted by
\[
f_A = \dashint_A f \dx = \frac{1}{|A|} \int_A f \dx .
\]

In many cases, it is preferable to consider medians instead of integral averages.
Let $0<s \leq 1$. Assume that $A \subset \mathbb{R}^n$ is a measurable set with $0<|A|<\infty$ and that $f:A\rightarrow [-\infty, \infty]$ is a measurable function.
A number $a\in\mathbb R$ is called an $s$-median of $f$ over $A$, if
\[
|\{x \in A: f(x) > a\}| \le s |A|
\quad\text{and}\quad
|\{x \in A: f(x) < a\}| \le (1-s)|A|.
\]
In general, the $s$-median is not unique. To obtain a uniquely defined notion, we consider the maximal $s$-median as in \cite{medians}.

\begin{definition}
\label{maxmedian}
Let $0<s \leq 1$. Assume that $A \subset \mathbb{R}^n$ is a measurable set with $0<|A|<\infty$ and that $f:A\rightarrow [-\infty, \infty]$ is a measurable function.
The maximal $s$-median of $f$ over $A$ is defined as
\[
m_f^s(A) = \inf \{a \in \mathbb{R} : |\{x \in A: f(x) > a\}| < s |A| \} .
\]
\end{definition}

The maximal $s$-median of a function is an $s$-median~\cite{medians}.
In the next lemma, we list the basic properties of the maximal $s$-median.
We refer to~\cite{myyrylainen} where the properties are proven in metric measure spaces.
The arguments are identical for Euclidean spaces.
The proofs of properties (i), (ii), (v), (vii), (viii) and (ix) can also be found in~\cite{medians}.
In addition, most of these properties are listed without proofs in~\cite{approxquasi, approxholder}.

\begin{lemma}
\label{medianprops}
Let $0<s \leq 1$. Assume that $A \subset \mathbb{R}^n$ is a measurable set with $0<|A|<\infty$ and that $f,g:A\rightarrow [-\infty, \infty]$ is a measurable function.
The maximal $s$-median has the following properties.
\begin{enumerate}[(i),topsep=5pt,itemsep=5pt]

\item$m_f^{s'}(A) \leq m_f^s(A)$ for $s \leq s'$.

\item $m_f^s(A) \leq m_g^s(A)$ whenever $f\leq g$ $\mu$-almost everywhere in $A$.

\item If $A \subset A'$ and $|A'| \leq c |A|$ with some $c \geq 1$, then $m_f^s(A) \leq m_f^{s/c}(A')$.

\item $m_{\varphi \circ f}^s(A) = \varphi(m_f^s(A))$ for an increasing continuous function $\varphi: f(A) \to [-\infty, \infty]$.

\item $m_f^s(A) + c = m_{f+c}^s(A)$ for $c \in \mathbb{R}$.

\item $m_{cf}^s(A) = c \, m_f^s(A)$ for $c > 0$.

\item $|m_{f}^s(A)| \leq m_{|f|}^{\min\{s,1-s\}}(A)$.

\item $m_{f+g}^s(A) \leq m_f^{t_1}(A) + m_g^{t_2}(A)$ whenever $t_1 + t_2 \leq s$.

\item For $f \in L^p(A)$ and $p>0$, \[ m_{|f|}^s(A) \leq \biggl( s^{-1} \dashint_A |f|^p \dx \biggr)^{\frac{1}{p}}. \]

\item If $A_i$ are pairwise disjoint for every $i \in \mathbb{N}$, then
\[
\inf_{i} m_f^s(A_i) \leq m_f^s\Bigl(\bigcup_{i=1}^\infty A_i\Bigr) \leq \sup_{i} m_f^s(A_i) .
\]

\end{enumerate}
\end{lemma}

\begin{remark}
Assume that $0<s\leq\frac12$. Then property (vii) assumes a slightly simpler form
\[
|m_{f}^s(A)| \leq m_{|f|}^{\min\{s,1-s\}}(A) = m_{|f|}^s(A) ,
\]
since 
\[
m_{|f|}^{1-s}(A) \leq m_{|f|}^s(A)
\]
for $0<s\leq\frac12$.
\end{remark}

A cube $Q$ is a bounded interval in $\mathbb R^n$, with sides parallel to the coordinate axes and equally long, that is,
$Q=[a_1,b_1] \times\dots\times [a_n, b_n]$
with $b_1-a_1=\ldots=b_n-a_n$. 
The side length of $Q$ is $l(Q)=b_1-a_1$.
In case we want to specify the center of a cube, we write 
$Q=Q(x,r)=\{y \in \mathbb R^n: |y_i-x_i| \leq r,\,i=1,\dots,n\}$
for $x\in\mathbb R^n$ and $r>0$. 
We consider closed cubes, but the results hold for open and half open cubes as well.

Let $Q_0\subset \mathbb R^n$ be a cube. 
The dyadic decomposition $\mathcal{D}(Q_0)$ of $Q_0$ is defined as
$\mathcal{D}(Q_0)=\bigcup_{j=0}^\infty\mathcal{D}_j(Q_0)$,
where each
$\mathcal{D}_j(Q_0)$ consists of $2^{jn}$ cubes $Q$, with pairwise disjoint  interiors and side length $l(Q)=2^{-j}l(Q_0)$,
such that $Q_0=\bigcup\{Q:Q\in\mathcal{D}_j(Q_0)\}$ for every $j\in \mathbb N_0$.
If $j\ge 1$ and $Q\in\mathcal{D}_j(Q_0)$, there exists a unique cube $Q'\in\mathcal{D}_{j-1}(Q_0)$
with $Q\subset Q'$. The cube $Q'$ is called the dyadic parent of $Q$, and $Q$ is a dyadic child of $Q'$.

We recall the Lebesgue differentiation theorem for medians.
The proof can be found in~\cite{medians}.

\begin{lemma}
\label{leb.diff.medians}
Let $f: \mathbb{R}^n \rightarrow [-\infty, \infty]$ be a measurable function which is finite almost everywhere in $\mathbb R^n$ and $0<s\leq 1$. Then 
\[
\lim_{i \to \infty}  m_f^s(Q_i) = f(x)
\]
for almost every $x \in \mathbb{R}^n$, whenever $(Q_i)_{i\in\mathbb N}$ is a sequence of (dyadic) cubes containing $x$ such that $\lim_{i \to \infty} |Q_i| = 0$.

\end{lemma}

We discuss a Calder\'{o}n-Zygmund decomposition with medians instead of integral averages.
The proof is a simple modification of the corresponding argument for integral averages in~\cite{john_original}.

\begin{lemma}
\label{C-Z_euclidean}
Let $Q_0 \subset \mathbb{R}^n$ be a cube and $0<t\leq 1$. Assume that $f:Q_0\to[-\infty,\infty]$ is a measurable function.
For every $\lambda\ge m_{|f|}^{t}(Q_0)$, there exist dyadic cubes $Q_i \in\mathcal D(Q_0)$, $i\in\mathbb N$,  with pairwise disjoint interiors, such that
\begin{enumerate}[(i),topsep=5pt,itemsep=5pt]

\item $ m_{|f|}^{t}(Q_i) > \lambda $,

\item $ m_{|f|}^{t}(Q_i') \leq \lambda $ where $Q_i'$ is the dyadic parent of $Q_i$,

\item $|f(x)| \leq \lambda$ for almost every $x \in Q_0 \setminus \bigcup_{i=1}^\infty Q_i $.

\end{enumerate}
The collection $\{Q_i\}_{i\in\mathbb N}$ is called the Calder\'{o}n-Zygmund cubes in $Q_0$ at level~$\lambda$. 
\end{lemma}

\begin{proof}
Consider the collection 
\[
\mathcal F_\lambda=\{Q\in\mathcal D(Q_0):m_{|f|}^{t}(Q) > \lambda\}.
\]
For every $x\in\bigcup_{Q \in \mathcal F_\lambda} Q$, there exists a cube $Q\in \mathcal F_\lambda$ with $x\in Q$ and $m_{|f|}^{t}(Q) > \lambda$.
It follows that there exists a unique maximal cube $Q_x\in \mathcal F_\lambda$ with $x\in Q_x$ and $m_{|f|}^{t}(Q_x) > \lambda$.
Maximality means that if $Q_x\subsetneq Q\in\mathcal{D}(Q_0)$, then $m_{|f|}^{t}(Q)\le\lambda$. 
Let $\{Q_i\}_{i\in\mathbb N}$ be the subcollection of $\mathcal F_\lambda$ of such maximal cubes.
If $Q_x=Q_0$ for some $x\in Q_0$, then $\mathcal F_\lambda = \{Q_0\}$
and there are no cubes $Q\in\mathcal{D}(Q_0)$ with $Q_x\subsetneq Q$. 
This happens if and only if $\lambda<m_{|f|}^{t}(Q_0)$, 
which contradicts the assumption $\lambda\ge m_{|f|}^{t}(Q_0)$. 

For two dyadic subcubes of $Q_0$, it holds that either one is contained in the other or the cubes have pairwise disjoint interiors. 
Thus, the collection $\{Q_i\}_{i\in\mathbb N}$ consists of cubes with pairwise disjoint interiors with
$m_{|f|}^{t}(Q_i) > \lambda$, $i\in\mathbb N$. This proves (i).
By maximality, it holds that $m_{|f|}^{t}(Q_i') \leq \lambda$ for every $i\in\mathbb N$, where $Q_i'$ is the dyadic parent of $Q_i$.
This implies (ii).
To prove (iii), assume that $x \in Q_0 \setminus \bigcup_{i=1}^\infty Q_i$. 
We have $m_{|f|}^{t}(Q) \leq \lambda$ for every dyadic subcube $Q$ of $Q_0$ containing $x$. 
Hence, there exist a decreasing sequence of dyadic subcubes $Q_k$ such that $x \in Q_k$ for every $k \in \mathbb{N}$ and $Q_{k+1} \subsetneq Q_k$.
The Lebesgue differentiation theorem for medians (Lemma~\ref{leb.diff.medians}) implies that
\[
|f(x)| = \lim_{k \to \infty} m_{|f|}^{t}(Q_k) \leq \lambda .
\]
for almost every point $x \in Q_0 \setminus \bigcup_{i=1}^\infty Q_i $.
\end{proof}

\section{John--Nirenberg inequality with medians}

This section discusses the John--Nirenberg inequality for median-type John--Nirenberg spaces.

\begin{definition}
\label{Def_medJN_eucli}
Let $Q_0 \subset \mathbb{R}^n$ be a cube, $1<p<\infty$ and $0<s \leq\frac12$, and assume that $f:Q_0\to[-\infty,\infty]$ is a measurable function. 
We say that $f$ belongs to the median-type dyadic John--Nirenberg space $JN^d_{p,0,s}(Q_0)$ if
\[
\norm{f}_{JN^d_{p,0,s}(Q_0)}^p = \sup \sum_{i=1}^{\infty} |Q_i| \bigl(\inf_{c_i \in \mathbb{R}} m_{|f-c_i|}^s (Q_i) \bigr)^p < \infty ,
\]
where the supremum is taken over countable collections $\{Q_i\}_{i\in\mathbb N}$ of pairwise disjoint dyadic subcubes of $Q_0$.
\end{definition}

The constants $c_i$ in the definition of $JN^d_{p,0,s}$ can be replaced by maximal $t$-medians with $0<s\leq t \leq \tfrac{1}{2}$.
A simple proof can be found in~\cite{myyrylainen}. For more on the median-type John--Nirenberg space, see~\cite{myyrylainen}.

\begin{lemma}
\label{lemma_med1_eucli}
Let $Q_0 \subset \mathbb{R}^n$ be a cube and assume that $f:Q_0\to[-\infty,\infty]$ is a measurable function. It holds that
\[
\norm{f}_{JN^d_{p,0,s}(Q_0)}^p \leq \sup \sum_{i=1}^{\infty} |Q_i| \bigl( m_{|f-m_f^t(Q_i)|}^s (Q_i) \bigr)^p \leq 2^p \norm{f}_{JN^d_{p,0,s}(Q_0)}^p ,
\]
whenever $0<s\leq t \leq\frac12$.

\end{lemma}

\begin{definition}
Let $Q_0 \subset \mathbb{R}^n$ be a cube and $0<t\leq 1$, and assume that $f:Q_0\to[-\infty,\infty]$ is a measurable function. 
The median-type dyadic maximal function is defined by
\[
\mathcal{M}^{d,t}_{Q_0}f(x) = \sup_{Q \ni x} m_{|f|}^{t}(Q) ,
\]
where the supremum is taken over all dyadic subcubes $Q\in\mathcal D(Q_0)$ with $x \in Q$.
\end{definition}

The following good-$\lambda$ inequality is the main ingredient in the proof of the John--Nirenberg inequality.

\begin{lemma}
\label{goodlambd_eucli}
Let $0< t \leq \tfrac{1}{2^{n+1}} $, $K>1$ and $f \in JN^d_{p,0,s}(Q_0)$ for some $0<s \leq \tfrac{t}{2K^p}$, and assume that 
$m_{|f|}^{t}(Q_0) \leq \lambda$.
Then
\[
|E_{K\lambda}(Q_0)| \leq \frac{2^p}{(K-1)^p} \frac{\norm{f}_{JN^d_{p,0,s}}^p}{\lambda^p} + \frac{1}{2K^p} |E_{\lambda}(Q_0)| ,
\]
where $E_{\lambda}(Q_0) = \{ x \in Q_0 : \mathcal{M}^{d,t}_{Q_0}f(x) > \lambda \}$.
\end{lemma}

\begin{proof}
We apply the Calder\'{o}n-Zygmund decomposition (Lemma~\ref{C-Z_euclidean}) for $f$ in $Q_0$ at levels $\lambda$ and $K\lambda$ to obtain collections of cubes
$\{Q_{i,\lambda}\}_{i\in\mathbb N}$ and $\{Q_{j,K\lambda}\}_{j\in\mathbb N}$ such that
\[
E_\lambda(Q_0) = \bigcup_{i=1}^\infty Q_{i,\lambda} \quad \text{and} \quad E_{K \lambda}(Q_0) = \bigcup_{j=1}^\infty Q_{j,K\lambda} .
\]
Denote
\[
J_i = \left\{ j \in \mathbb{N}: Q_{j,K\lambda} \subset Q_{i,\lambda} \right\}
\]
for every $i \in \mathbb{N}$, and
\[
I = \Bigl\{ i\in \mathbb{N}: \lvert Q_{i,\lambda} \rvert \leq 2 K^{p} \big\lvert \bigcup_{j \in J_i} Q_{j,K \lambda}\big\rvert  \Bigr\} .
\]
Since each $Q_{j,K\lambda}$ is contained in some $Q_{i,\lambda}$, we get the partition
\[
\bigcup_{j=1}^\infty Q_{j,K\lambda} = \bigcup_{i=1}^\infty \bigcup_{j \in J_i} Q_{j,K\lambda} .
\]
By Lemma~\ref{medianprops} (ii), (v), (vii), (iii) and Lemma~\ref{C-Z_euclidean} (ii) in this order, we obtain
\begin{align*}
m_{|f-m_f^{2^n t}(Q_{i,\lambda})|}^{t}(Q_{j,K\lambda}) &\geq m_{|f|}^{t}(Q_{j,K\lambda}) - |m_f^{2^n t}(Q_{i,\lambda})|
\geq m_{|f|}^{t}(Q_{j,K\lambda}) - m_{|f|}^{2^n t}(Q_{i,\lambda}) \\
&\geq m_{|f|}^{t}(Q_{j,K\lambda}) - m_{|f|}^{t}(Q'_{i,\lambda})
\geq K\lambda - \lambda = (K-1) \lambda ,
\end{align*}
where $Q'_{i,\lambda}$ is the parent cube of $Q_{i,\lambda}$.
Since $Q_{j,K\lambda}$ are pairwise disjoint, property (x) of Lemma~\ref{medianprops} implies that
\[
m_{|f-m_f^{2^n t}(Q_{i,\lambda})|}^{t}\big(\bigcup_{j \in J_i} Q_{j,K\lambda}\big) \geq (K-1)\lambda .
\]
By applying Lemma~\ref{medianprops}~(iii), we get
\begin{align*}
\sum_{j \in J_i} |Q_{j,K\lambda}| &\leq |Q_{i,\lambda}| 
\leq  \frac{1}{(K-1)^p \lambda^p} |Q_{i,\lambda}| \biggl( m_{|f-m_f^{2^n t}(Q_{i,\lambda})|}^{t}\big(\bigcup_{j \in J_i} Q_{j,K\lambda}\big) \biggr)^p \\
&\leq \frac{1}{(K-1)^p \lambda^p} |Q_{i,\lambda}| \Bigl(  m_{|f-m_f^{2^n t}(Q_{i,\lambda})|}^{t/2 K^p}(Q_{i,\lambda}) \Bigr)^p
\end{align*}
for $i \in I$.
Hence, by summing over all indices $i \in I$, we obtain
\[
\sum_{i \in I} \sum_{j \in J_i} |Q_{j,K\lambda}| 
\leq \frac{1}{(K-1)^p\lambda^p} \sum_{i \in I} |Q_{i,\lambda}| \Bigl(  m_{|f-m_f^{2^n t}(Q_{i,\lambda})|}^{t/2 K^p}(Q_{i,\lambda})\Bigr)^p 
\leq \frac{2^p}{(K-1)^p} \frac{\norm{f}_{JN^d_{p,0,s}}^p}{\lambda^p}  ,
\]
where in the last inequality we used Lemma~\ref{lemma_med1_eucli} with $t \leq \tfrac{1}{2^{n+1}}$ and $0 < s \leq \tfrac{t}{2 K^p}$.

On the other hand, if $i \notin I$, we have
\[
\sum_{j \in J_i} |Q_{j,K\lambda}| \leq \frac{1}{2K^p} |Q_{i,\lambda}| .
\]
Summing over all indices $i \notin I$, it follows that
\[
\sum_{i \notin I} \sum_{j \in J_i} |Q_{j,K\lambda}| \leq \frac{1}{2K^p} \sum_{i \notin I} |Q_{i,\lambda}| \leq \frac{1}{2K^p} |E_{\lambda}(Q_0)| .
\]
By combining the cases $i\in I$ and $i \notin I$, we conclude that
\begin{align*}
|E_{K \lambda}(Q_0)| = \sum_{i = 1}^\infty \sum_{j \in J_i} |Q_{j,K\lambda}| \leq \frac{2^p}{(K-1)^p} \frac{\norm{f}_{JN^d_{p,0,s}}^p}{\lambda^p} + \frac{1}{2K^p} |E_{\lambda}(Q_0)| .
\end{align*}

\end{proof}

We are ready to prove the John--Nirenberg inequality for $JN^d_{p,0,s}$ which implies that $JN^d_{p,0,s}(Q)$ is contained in $L^{p,\infty}(Q)$ for all cubes $Q \subset \mathbb{R}^n$.

\begin{theorem}
\label{thm_JN_eucli}
Let $0<s \leq \tfrac{1}{2^{n+3}}$ and $s \leq r \leq \frac{1}{2}$.
If $f \in JN^d_{p,0,s}(Q_0)$, then there exists a constant $c=c(p)$ such that for every $\lambda > 0$ we have
\[
|\{ x \in Q_0: |f(x)-m_f^{r}(Q_0)| > \lambda \}| \leq c \frac{\norm{f}_{JN^d_{p,0,s}(Q_0)}^p}{\lambda^p} .
\]
\end{theorem}

\begin{proof}
Since $f \in JN^d_{p,0,s}(Q_0)$, Lemma~\ref{lemma_med1_eucli} implies that
\[
|Q_0|^\frac{1}{p} m_{|f-m_f^{r}(Q_0)|}^{t}(Q_0) \leq 2 \norm{f}_{JN^d_{p,0,s}} ,
\]
where $t = \frac{1}{2^{n+1}}$ and $s\leq r \leq \frac{1}{2}$.
Therefore, the condition in Lemma~\ref{goodlambd_eucli} holds for $|f-m_f^{r}(Q_0)|$ with the choice
\[
\lambda_0 = \frac{2 \norm{f}_{JN^d_{p,0,s}}}{|Q_0|^\frac{1}{p}} .
\]
For $0 < \lambda \leq \lambda_0$, we have
\begin{align*}
|\{ x \in Q_0: |f(x)-m_f^{r}(Q_0)| > \lambda \}| \leq |Q_0| = 2^p \frac{\norm{f}_{JN^d_{p,0,s}}^p}{\lambda_0^p} \leq 2^p \frac{\norm{f}_{JN^d_{p,0,s}}^p}{\lambda^p} .
\end{align*}

Assume then that $\lambda > \lambda_0$.
Let $K = 2^{\frac1p}$ and choose $N \in \mathbb{N}$ such that
\[
K^N \lambda_0 < \lambda \leq K^{N+1} \lambda_0 .
\]
We have
\[
|\{ x \in Q_0: |f(x)-m_f^{r}(Q_0)| > \lambda \}| 
\leq |\{ x \in Q_0: |f(x)-m_f^{r}(Q_0)| > K^N \lambda_0 \}| \\
\leq |E_{K^N \lambda_0}(Q_0)| ,
\]
where the last inequality follows from Lemma~\ref{C-Z_euclidean}~(iii).
We claim that
\[
|E_{K^m \lambda_0}(Q_0)| \leq c_0 \frac{\norm{f}_{JN^d_{p,0,s}}^p}{(K^{m} \lambda_0)^p} 
\]
for every $m=0,1,\dots,N$, where $c_0 = 2^{p+1} K^p(K-1)^{-p}$.
We prove the claim by induction. First, observe that the claim holds for $m = 0$, since
\[
|E_{ \lambda_0}(Q_0)| \leq |Q_0| = 2^p \frac{ \norm{f}_{JN^d_{p,0,s}}^p}{\lambda_0^p} \leq c_0 \frac{ \norm{f}_{JN^d_{p,0,s}}^p}{\lambda_0^p}.
\]
Assume then that the claim holds for $k \in \{0,1,\dots,N-1\}$, that is,
\[
|E_{K^k \lambda_0}(Q_0)| \leq c_0 \frac{\norm{f}_{JN^d_{p,0,s}}^p}{(K^{k} \lambda_0)^p} .
\]
This together with Lemma~\ref{goodlambd_eucli} for $K^k \lambda_0$ implies the claim for $k+1$:
\begin{align*}
|E_{K^{k+1} \lambda_0}(Q_0)| &\leq \frac{2^p }{(K-1)^p } \frac{\norm{f}_{JN^d_{p,0,s}}^p}{(K^k \lambda_0)^p} + \frac{1}{2K^p} |E_{K^k \lambda_0}(Q_0)| \\
&\leq \frac{2^p }{(K-1)^p } \frac{\norm{f}_{JN^d_{p,0,s}}^p}{(K^k \lambda_0)^p} + \frac{c_0}{2K^p} \frac{\norm{f}_{JN^d_{p,0,s}}^p}{(K^{k} \lambda_0)^p} \\
&= \bigg( \frac{2^p K^p}{(K-1)^p} + \frac{c_0}{2} \bigg) \frac{\norm{f}_{JN^d_{p,0,s}}^p}{(K^{k+1} \lambda_0)^p} 
= c_0 \frac{\norm{f}_{JN^d_{p,0,s}}^p}{(K^{k+1} \lambda_0)^p} .
\end{align*}
Hence, the claim holds for $k+1$.

We conclude that
\[
|\{ x \in Q_0: |f(x)-m_f^{r}(Q_0)| > \lambda \}| 
\leq  c_0 \frac{\norm{f}_{JN^d_{p,0,s}}^p}{(K^N \lambda_0)^p} 
= c_0 K^p \frac{\norm{f}_{JN^d_{p,0,s}}^p}{(K^{N+1} \lambda_0)^p} 
\leq c \frac{\norm{f}_{JN^d_{p,0,s}}^p}{\lambda^p} ,
\]
with $c = c_0  K^p = 2^{p+1} K^{2p}(K-1)^{-p} =2^{p+3}(2^\frac{1}{p} -1)^{-p}$.
\end{proof}

As an application of the John--Nirenberg inequality (Theorem~\ref{thm_JN_eucli}), we discuss the connection between the John--Nirenberg spaces with medians and integral averages.

\begin{definition}
\label{Def_JN^dp}
Let $Q_0 \subset \mathbb{R}^n$ be a cube and $1<p<\infty$. We say that $f \in L^1(Q_0)$ belongs to the dyadic John--Nirenberg space $JN^d_{p}(Q_0)$ if
\[
\norm{f}_{JN^d_{p}(Q_0)}^p = \sup \sum_{i=1}^\infty |Q_i| \biggl( \dashint_{Q_i} |f-f_{Q_i}| \dx \biggr)^p < \infty,
\]
where the supremum is taken over countable collections  $\{Q_i\}_{i\in\mathbb N}$ of pairwise disjoint dyadic subcubes of $Q_0$.
\end{definition}

As a corollary of Theorem~\ref{thm_JN_eucli}, the median-type dyadic John--Nirenberg space coincides with the dyadic John--Nirenberg space with integral averages.
In particular, it follows that all results for the dyadic John--Nirenberg spaces with integral averages also hold for the median-type dyadic John--Nirenberg spaces and vice versa.
We note that Theorem~\ref{thm_JN_eucli} also holds for the John--Nirenberg space over all subcubes instead of dyadic subcubes of $Q_0$.
Thus, the corollary below also holds for the standard John--Nirenberg spaces.

\begin{corollary}
\label{equivalence_euclidean}
Let $1<p<\infty$ and $0< s \leq \frac{1}{2^{n+3}}$.
It holds that
\[
s \norm{f}_{JN^d_{p,0,s}(Q_0)} \leq \norm{f}_{JN^d_{p}(Q_0)} \leq \frac{2cp}{p-1}  \norm{f}_{JN^d_{p,0,s}(Q_0)},
\]
where $c$ is the constant from Theorem~\ref{thm_JN_eucli}.
\end{corollary}

\begin{proof}
Let $\{Q_i\}_{i\in\mathbb N}$ be a collection of pairwise disjoint dyadic subcubes of $Q_0$.
The first inequality follows in a straightforward manner from Lemma~\ref{medianprops}~(ix).
For the second inequality, we use Cavalieri's principle together with Theorem~\ref{thm_JN_eucli} to obtain
\begin{align*}
\int_{Q_i}& |f-m_f^{s}(Q_i)| \dx 
= \int_0^\infty |\{ x \in Q_i: |f-m_f^{s}(Q_i)| > \lambda \}| \dla \\
&\leq \int_{|Q_i|^{-\frac{1}{p}} \norm{f}_{JN^d_{p,0,s}(Q_i)}}^\infty c \lambda^{-p} \norm{f}_{JN^d_{p,0,s}(Q_i)}^p \dla  
+ \int_0^{|Q_i|^{-\frac{1}{p}} \norm{f}_{JN^d_{p,0,s}(Q_i)}} |Q_i| \dla \\
&= \frac{c}{p-1} |Q_i|^{1-\frac{1}{p}} \norm{f}_{JN^d_{p,0,s}(Q_i)} + |Q_i|^{1-\frac{1}{p}} \norm{f}_{JN^d_{p,0,s}(Q_i)} \\
&\leq \frac{cp}{p-1} |Q_i|^{1-\frac{1}{p}} \norm{f}_{JN^d_{p,0,s}(Q_i)} ,
\end{align*}
where  $c$ is the constant from Theorem~\ref{thm_JN_eucli}.
This implies that
\begin{align*}
\sum_{i=1}^\infty |Q_i| \biggl(\inf_{c_i} \dashint_{Q_i}|f-c_i| \dx \biggr)^p 
&\leq \sum_{i=1}^\infty |Q_i| \biggl( \dashint_{Q_i} |f-m_f^{s}(Q_i)| \dx \biggr)^p \\
&\leq \lp \frac{cp}{p-1} \rp^p \sum_{i=1}^\infty \norm{f}_{JN^d_{p,0,s}(Q_i)}^p 
\leq \lp \frac{cp}{p-1} \rp^p \norm{f}_{JN^d_{p,0,s}(Q_0)}^p .
\end{align*}
Thus, it follows that
\[
\norm{f}_{JN^d_{p}(Q_0)} \leq \frac{2cp}{p-1}  \norm{f}_{JN^d_{p,0,s}(Q_0)}.
\]
\end{proof}

\section{The dyadic maximal function on $JN^d_p$}

In this section, we discuss the behavior of the Hardy--Littlewood maximal function on the John--Nirenberg space with integral averages as in Definition \ref{Def_JN^dp}.

\begin{definition}
Let $Q_0 \subset \mathbb{R}^n$ be a cube and assume that $f \in L^1(Q_0)$.
The dyadic maximal function of $f$ is defined by
\[
M^d_{Q_0}f(x) = \sup \dashint_{Q} |f(y)| \dy ,
\]
where the supremum is taken over all dyadic subcubes $Q\in\mathcal D(Q_0)$ with $x \in Q$.
\end{definition}

Let $f,g\in L^1(Q_0)$ and $x\in Q_0$. 
Using the definition, it is easy to show
that $M^d_{Q_0}f(x)\ge 0$, 
\[
M^d_{Q_0}(f+g)(x)\le M^d_{Q_0}f(x)+M^d_{Q_0}g(x),
\] 
and
\[
M^d_{Q_0}(af)(x)=\lvert a\rvert M^d_{Q_0}f(x)
\]
for every $a\in\mathbb R$.

The Calder\'{o}n-Zygmund decomposition with integral averages implies that
the dyadic maximal function satisfies the weak type estimate
\begin{equation}
\label{weaktype}
|\{x\in Q_0:M^d_{Q_0}f(x)>\lambda\}|
\le\frac1\lambda\int_{Q_0}|f(x)|\,dx
\end{equation}
for every $\lambda>0$ and is a bounded operator on $L^p(Q_0)$ with $1<p\le \infty$.
Moreover, the dyadic maximal operator is bounded on $BMO(Q_0)$; see \cite{bennett}.
We show that the dyadic maximal operator is bounded on the dyadic John--Nirenberg space.

\begin{theorem}\label{thm_mfbound}
Let $1 < p <\infty$ and assume that  $f \in JN^d_{p}(Q_0)$. Then there exists a constant $c = c(n,p)$ such that
\[
\lVert M^d_{Q_0}f\rVert_{JN^d_{p}(Q_0)} \leq c \norm{f}_{JN^d_{p}(Q_0)} .
\]

\end{theorem}

\begin{proof}

Let $\{Q_i\}_{i\in\mathbb N}$ be a collection of pairwise disjoint dyadic subcubes of $Q_0$.
Denote
\[
E_i = \{ x \in Q_i : M^d_{Q_0}f(x) = M^d_{Q_i}f(x) \},\quad i\in\mathbb N .
\]
For $x \in Q_i \setminus E_i$, the supremum in the definition of $M^d_{Q_0}f(x)$ is attained in a dyadic cube $Q_x \ni x$ that intersects $Q_0 \setminus Q_i$.
Since both $Q_x$ and $Q_i$ are dyadic subcubes of $Q_0$ and $x \in Q_i \cap Q_x$, it follows that $Q_i \subset Q_x$.
Since $Q_i \subset Q_x$ for every $x \in Q_i \setminus E_i$, the cube $Q_x$ for which the supremum in the maximal function is attained is the same cube for every $x \in Q_i \setminus E_i$.
Thus, for every $i\in\mathbb N$, there exists a constant $M_i$ such that $M^d_{Q_0}f(x)=M_i$ for every $x \in Q_i \setminus E_i$.
We observe that
\begin{align*}
M^d_{Q_i}f - (M^d_{Q_0}f)_{Q_i} 
\leq M^d_{Q_i}f - |f_{Q_i}|
= M^d_{Q_i}f - M^d_{Q_i}(f_{Q_i}) 
\leq M^d_{Q_i}(f-f_{Q_i}) .
\end{align*}
This implies that
\begin{align*}
\frac{1}{2} \int_{Q_i}&|M^d_{Q_0}f - ( M^d_{Q_0}f )_{Q_i} |\dx
= \int_{Q_i} \big( M^d_{Q_0}f - (M^d_{Q_0}f)_{Q_i} \big)^+\dx\\
&= \int_{E_i} \big( M^d_{Q_0}f - (M^d_{Q_0}f)_{Q_i} \big)^+\dx
+ \int_{Q_i \setminus E_i} \big( M^d_{Q_0}f - (M^d_{Q_0}f)_{Q_i} \big)^+\dx \\
&= \int_{E_i}  \big( M^d_{Q_i}f - (M^d_{Q_0}f)_{Q_i} \big)^+\dx  
+ \int_{Q_i \setminus E_i} \big( M_i - (M^d_{Q_0}f)_{Q_i} \big)^+\dx \\
&\leq \int_{E_i}  M^d_{Q_i}(f-f_{Q_i})\dx
 \leq \int_{Q_i}  M^d_{Q_i}(f-f_{Q_i})\dx ,
\end{align*}
where in the second last inequality we also used $M_i \leq (M^d_{Q_0}f)_{Q_i}$, $i\in\mathbb N$, which follows from $M_i\leq  M^d_{Q_0}f(x)$ for every $x \in Q_i$.
From the proof of the John--Nirenberg lemma~\cites[pp.~11--13]{aalto}[p.~7]{kinnunen}, we see that
\[
|\{ x \in Q_i : M^d_{Q_i}(f-f_{Q_i})(x) > \lambda \}| \leq c \frac{\norm{f}_{JN^d_{p}(Q_i)}^p}{\lambda^p} 
\]
for some constant $c = c(n,p)$.
Applying this together with Cavalieri's principle, we obtain
\begin{align*}
\int_{Q_i}& M^d_{Q_i}(f-f_{Q_i}) \dx = \int_0^\infty |\{ x \in Q_i: M^d_{Q_i}(f-f_{Q_i})(x) > \lambda \}| \dla \\
&\leq \int_{|Q_i|^{-\frac{1}{p}} \norm{f}_{JN^d_{p}(Q_i)}}^\infty c \lambda^{-p} \norm{f}_{JN^d_{p}(Q_i)}^p \dla 
+ \int_0^{|Q_i|^{-\frac{1}{p}} \norm{f}_{JN^d_{p}(Q_i)}} |Q_i| \dla \\
&= \frac{c}{p-1} |Q_i|^{1-\frac{1}{p}} \norm{f}_{JN^d_{p}(Q_i)} + |Q_i|^{1-\frac{1}{p}} \norm{f}_{JN^d_{p}(Q_i)} 
\leq \frac{cp}{p-1} |Q_i|^{1-\frac{1}{p}} \norm{f}_{JN^d_{p}(Q_i)}.
\end{align*}
Therefore, we can conclude that
\begin{align*}
\sum_{i=1}^\infty |Q_i| \biggl( \dashint_{Q_i}& |M^d_{Q_0}f - (M^d_{Q_0}f)_{Q_i} | \dx \biggr)^p 
\leq  2^p \sum_{i=1}^\infty |Q_i| \biggl( \dashint_{Q_i} M^d_{Q_0}(f-f_{Q_i})  \dx \biggr)^p \\
&\leq \lp \frac{2cp}{p-1} \rp^p \sum_{i=1}^\infty \norm{f}_{JN^d_{p}(Q_i)}^p 
\leq \lp \frac{2cp}{p-1} \rp^p \norm{f}_{JN^d_{p}(Q_0)}^p .
\end{align*}
Taking the supremum over all collections of $\{Q_i\}_{i\in\mathbb N}$, we get
\[
\lVert M^d_{Q_0}f\rVert_{JN^d_{p}(Q_0)}^p \leq \lp \frac{2cp}{p-1} \rp^p \norm{f}_{JN^d_{p}(Q_0)}^p .
\]

\end{proof}

By a similar argument as in the proof of Theorem \ref{thm_mfbound}, we obtain an $L^1$ result for the dyadic maximal function.
The weak type estimate \eqref{weaktype} is used instead of the John--Nirenberg inequality in the argument.

\begin{theorem}
\label{construct_dyaJNp}
Let $1 < p <\infty$ and assume that $f \in L^1(Q_0)$. Then there exists a constant $c = c(p)$ such that
\[
\lVert(M^d_{Q_0}f)^\frac{1}{p}\rVert_{JN^d_{p}(Q_0)}^p \leq c \norm{f}_{L^1(Q_0)} .
\]

\end{theorem}

\begin{proof}
We use the same notation as in the proof of Theorem \ref{thm_mfbound}.
Analogously, we observe that
\begin{align*}
M^d_{Q_i}f - \big(\big[(M^d_{Q_0}f)^\frac{1}{p}\big]_{Q_i}\big)^p 
\leq M^d_{Q_i}f - |f_{Q_i}|
= M^d_{Q_i}f - M^d_{Q_i}(f_{Q_i})
\leq M^d_{Q_i}(f-f_{Q_i}) ,
\end{align*}
since $|f_{Q_i}| \leq M^d_{Q_0}f(x)$ for every $x \in Q_i$.
This implies
\begin{align*}
\frac{1}{2} \int_{Q_i}& \big|(M^d_{Q_0}f)^\frac{1}{p} - \big[(M^d_{Q_0}f)^\frac{1}{p}\big]_{Q_i} \big|\dx
= \int_{Q_i} \big( (M^d_{Q_0}f)^\frac{1}{p} - \big[(M^d_{Q_0}f)^\frac{1}{p}\big]_{Q_i} \big)^+ \dx\\
&\leq \int_{Q_i} \big( M^d_{Q_0}f - \big(\big[(M^d_{Q_0}f)^\frac{1}{p}\big]_{Q_i}\big)^p \big)^\frac{1}{p}_+ \dx\\
&= \int_{E_i} \big( M^d_{Q_0}f - \big(\big[(M^d_{Q_0}f)^\frac{1}{p}\big]_{Q_i}\big)^p \big)^\frac{1}{p}_+ \dx
 + \int_{Q_i \setminus E_i} \big( M^d_{Q_0}f - \big(\big[(M^d_{Q_0}f)^\frac{1}{p}\big]_{Q_i}\big)^p \big)^\frac{1}{p}_+\dx \\
&= \int_{E_i}  \big( M^d_{Q_i}f - \big(\big[(M^d_{Q_0}f)^\frac{1}{p}\big]_{Q_i}\big)^p \big)^\frac{1}{p}_+\dx
 + \int_{Q_i \setminus E_i} \big( M_i - \big(\big[(M^d_{Q_0}f)^\frac{1}{p}\big]_{Q_i}\big)^p \big)^\frac{1}{p}_+\dx \\
&\leq \int_{E_i}  \big[ M^d_{Q_i}(f-f_{Q_i}) \big]^\frac{1}{p}\dx
 \leq \int_{Q_i} \big[ M^d_{Q_i}(f-f_{Q_i}) \big]^\frac{1}{p}\dx ,
\end{align*}
where in the second last inequality we also used
\[
M_i \leq \big(\big[(M^d_{Q_0}f)^\frac{1}{p}\big]_{Q_i}\big)^p,
\quad i\in\mathbb N .
\]
Applying Cavalieri's principle together with the weak type estimate \eqref{weaktype} for the dyadic maximal operator, we obtain
\begin{align*}
\int_{Q_i}& \big[ M^d_{Q_i}(f-f_{Q_i}) \big]^\frac{1}{p} \dx 
= \frac{1}{p} \int_0^\infty \lambda^{\frac{1}{p}-1} |\{ x \in Q_i: M^d_{Q_i}(f-f_{Q_i})(x) > \lambda \}| \dla \\
&\leq \frac{1}{p} \int_{\lVert f - f_{Q_i}\rVert_{L^1(Q_i)} / |Q_i|}^\infty \lambda^{\frac{1}{p}-2} \lVert f - f_{Q_i}\rVert_{L^1(Q_i)} \dla 
 + \frac{1}{p} \int_0^{\lVert f - f_{Q_i}\rVert_{L^1(Q_i)} / |Q_i|} \lambda^{\frac{1}{p}-1} |Q_i| \dla \\
&= \frac{1}{p-1} |Q_i|^{1-\frac{1}{p}} \lVert f - f_{Q_i}\rVert_{L^1(Q_i)}^\frac{1}{p} +  |Q_i|^{1-\frac{1}{p}}\lVert f - f_{Q_i}\rVert_{L^1(Q_i)}^\frac{1}{p} \\
&= \frac{p}{p-1} |Q_i|^{1-\frac{1}{p}} \lVert f - f_{Q_i}\lVert_{L^1(Q_i)}^\frac{1}{p}
\leq 2^\frac{1}{p} \frac{ p}{p-1} |Q_i|^{1-\frac{1}{p}} \norm{f}_{L^1(Q_i)}^\frac{1}{p}.
\end{align*}
Therefore, we can conclude that
\begin{align*}
\sum_{i=1}^\infty |Q_i|\biggl( \dashint_{Q_i} \big|( & M^d_{Q_0}f)^\frac{1}{p} - \big[(M^d_{Q_0}f)^\frac{1}{p}\big]_{Q_i} \big| \dx \biggr)^p
\leq  2^p \sum_{i=1}^\infty |Q_i| \biggl( \dashint_{Q_i} \big[ M^d_{Q_i}(f-f_{Q_i}) \big]^\frac{1}{p}  \dx\biggr)^p \\
&\leq 2^{p+1} \lp \frac{p}{p-1} \rp^p \sum_{i=1}^\infty \norm{f}_{L^1(Q_i)}
\leq 2^{p+1} \lp \frac{p}{p-1} \rp^p \norm{f}_{L^1(Q_0)} .
\end{align*}
Taking the supremum over all collections of $\{Q_i\}_{i\in\mathbb N}$, we obtain
\[
\lVert(M^d_{Q_0}f)^\frac{1}{p}\rVert_{JN^d_{p}(Q_0)}^p \leq 2^{p+1} \lp \frac{p}{p-1} \rp^p \norm{f}_{L^1(Q_0)} .
\]
\end{proof}

\begin{corollary}
Let $1<p<\infty$ and assume that  $f \in L^1(Q_0) \setminus L \log^+ L(Q_0)$. Then $(M^d_{Q_0}f)^\frac{1}{p} \in JN^d_{p}(Q_0) \setminus L^p(Q_0)$.
\end{corollary}

\begin{proof}
Since $f \in L^1(Q_0)$, it follows that $(M^d_{Q_0}f)^\frac{1}{p} \in JN^d_{p}(Q_0)$ by Theorem~\ref{construct_dyaJNp}.
We know that a function $g$ is in $L \log^+ L(Q_0)$ if and only if $M^d_{Q_0}g$ is in $L^1(Q_0)$~\cite{stein_LlogL}.
Therefore, we have $M^d_{Q_0}f \notin L^1(Q_0)$, and thus $(M^d_{Q_0}f)^\frac{1}{p} \notin L^p(Q_0)$.
\end{proof}

This provides a method to construct functions in $JN_p^d \setminus L^p$. 
Consider a one-dimensional example. Let $I_0 = (0,\frac18)$ and $f:I_0\to\mathbb R$, 
\[
f(x) = \frac{\chi_{(0,1/2)}(x)}{x (\log x)^2 } .
\]
It holds that $f \in L^1(I_0) \setminus L \log^+ L(I_0)$.
Since $f$ is monotone on $I_0$, its maximal function is monotone on $I_0$ as well.
Hence, it cannot be in the standard John--Nirenberg space $JN_p(I_0)$, $1<p<\infty$, since $JN_p(I_0) = L^p(I_0)$ for monotone functions~\cite{korte}.
Thus, we have $(M^d_{I_0}f)^\frac{1}{p} \in JN^d_p(I_0) \setminus JN_p(I_0)$ and $(M^d_{I_0}f)^\frac{1}{p}\in JN^d_{p}(I_0) \setminus L^p(I_0)$.

\section{Completeness of $JN^d_p$}

The standard $BMO$ is complete with respect to the $BMO$ seminorm; see \cite{neri}.
We prove that the dyadic John--Nirenberg space is complete. 
Our proof also works for the standard John--Nirenberg space $JN_p$; see~\eqref{original_JNp}.

\begin{theorem}
\label{completeness}
Let $Q_0 \subset \mathbb{R}^n$ be a cube.
The space $JN^d_p(Q_0)$ is complete with respect to the seminorm in Definition \ref{Def_JN^dp}. 
\end{theorem}

\begin{proof}
Assume that $(f_j)_{j\in\mathbb N}$ is a Cauchy sequence in $JN^d_p(Q_0)$ and let $\varepsilon >0$. 
There exists $j_\varepsilon \in \mathbb{N}$ such that 
\[
\lVert f_j - f_k\rVert_{JN^d_{p}(Q_0)} < \varepsilon
\quad\text{whenever}\quad j,k \geq j_\varepsilon.
\]
Consider a collection $\{Q_i\}_{i\in\mathbb N}$ of pairwise disjoint dyadic cubes $Q_i \subset Q_0$, $i\in\mathbb N$.
Let
\[
g_j = \sum_{i=1}^\infty \chi_{Q_i} (f_j - (f_j)_{Q_i})
\]
and observe that
\begin{align*}
\sum_{i=1}^\infty |Q_i| \biggl(\dashint_{Q_i} |g_{j} - g_k| \dx \biggr)^p 
&= \sum_{i=1}^\infty |Q_i| \biggl( \dashint_{Q_i} |f_{j} - (f_{j})_{Q_i} - (f_k - (f_k)_{Q_i})| \dx \biggr)^p \\
&= \sum_{i=1}^\infty |Q_i| \biggl( \dashint_{Q_i} |f_{j} - f_{k} - (f_j - f_k)_{Q_i}| \dx\biggl)^p \\
&\leq \lVert f_{j} - f_{k}\rVert_{JN^d_{p}(Q_0)}^p  .
\end{align*}
Choose a subsequence $(f_{j_m})_{m\in\mathbb N}$ such that
\[
\lVert f_{j_{m+1}} - f_{j_m}\rVert_{JN^d_{p}(Q_0)} < \frac{1}{2^m}
\]
for every $m \in \mathbb{N}$.
Denote
\[
h_l = \sum_{m=1}^l |g_{j_{m+1}} - g_{j_m}| \quad \text{and} \quad h = \sum_{m=1}^\infty |g_{j_{m+1}} - g_{j_m}| .
\]
It then holds that $ \lim_{l \to \infty} h_l = h $.
By using Fatou's lemma and Minkowski's inequality, we obtain
\begin{align*}
\biggl(\sum_{i=1}^\infty |Q_i|\biggl( \dashint_{Q_i} |h| \dx \biggr)^p\biggr)^\frac{1}{p} 
&\leq \liminf_{l \to \infty} \biggl( \sum_{i=1}^\infty  |Q_i|\biggl(\dashint_{Q_i} |h_l| \dx\biggr)^p\biggr)^\frac{1}{p} \\
&\leq \liminf_{l \to \infty} \sum_{m=1}^l\biggl(\sum_{i=1}^\infty  |Q_i| \biggl(\dashint_{Q_i} |g_{j_{m+1}} - g_{j_m}| \dx\biggr)^p\biggr)^\frac{1}{p} \\
&\leq \sum_{m=1}^\infty \lVert f_{j_{m+1}} - f_{j_m}\rVert_{JN^d_{p}(Q_0)} 
\leq \sum_{m=1}^\infty \frac{1}{2^m} = 1 .
\end{align*}
Thus, $h \in L^1(Q_i)$ for every $i\in\mathbb N$ and consequently $h(x) < \infty$ for almost every $x \in \bigcup_{i=1}^\infty  Q_i$.
This implies that the series in
\[
g = g_{j_1} + \sum_{m=1}^\infty (g_{j_{m+1}} - g_{j_m})
\]
converges absolutely for almost every $x \in \bigcup_{i=1}^\infty  Q_i$.
Hence, we have
\begin{align*}
g &= g_{j_1} + \sum_{m=1}^\infty (g_{j_{m+1}} - g_{j_m})
= \lim_{l \to \infty} \Bigl(g_{j_1} + \sum_{m=1}^{l-1} (g_{j_{m+1}} - g_{j_m}) \Bigr) 
= \lim_{l \to \infty} g_{j_l} = \lim_{m \to \infty} g_{j_{m}}
\end{align*}
for almost every $x \in \bigcup_{i=1}^\infty  Q_i$.
By Fatou's lemma, we obtain
\begin{equation}
\label{eq1}
\begin{split}
\sum_{i=1}^\infty |Q_i| \biggl(\dashint_{Q_i} |g - g_j| \dx\biggr)^p 
&\leq \liminf_{m \to \infty} \sum_{i=1}^\infty |Q_i| \biggl( \dashint_{Q_i} |g_{j_m} - g_j| \dx \biggr)^p \\
&\leq \liminf_{m \to \infty}\lVert f_{j_m} - f_j\rVert_{JN^d_{p}(Q_0)}^p < \varepsilon^p ,
\end{split}
\end{equation}
whenever $j \geq j_\varepsilon$.

Consider the collection consisting only of the cube $Q_0$.
Then as above, we have
\[
g_{j}^{Q_0} = f_{j} - (f_{j})_{Q_0}
\]
and
\[
g^{Q_0} = g^{Q_0}_{j_1} + \sum_{m=1}^\infty (g^{Q_0}_{j_{m+1}} - g^{Q_0}_{j_m}) = \lim_{m \to \infty} g_{j_m}^{Q_0}
\]
almost everywhere in $Q_0$.
Similarly, we obtain
\begin{align*}
|Q_0| \biggl( \dashint_{Q_0} |g^{Q_0} - g_j^{Q_0}| \dx\biggr)^p 
&\leq \liminf_{m \to \infty}  |Q_0| \biggl( \dashint_{Q_0} |g_{j_m}^{Q_0} - g_j^{Q_0}| \dx\biggr)^p \\
&\leq \liminf_{m \to \infty}\lVert f_{j_m} - f_j\rVert_{JN^d_{p}(Q_0)}^p < \varepsilon^p ,
\end{align*}
whenever $j \geq j_\varepsilon$.
We see that $g^{Q_0} \in L^1(Q_0)$ and $g_{j}^{Q_0} = f_{j} - (f_{j})_{Q_0} \to g^{Q_0}$ in $L^1(Q_0)$ as $j\to\infty$, and thus
\begin{align*}
(f_{j_m})_{Q_i} - (f_{j_m})_{Q_0} = \dashint_{Q_i} \big( f_{j_m} - (f_{j_m})_{Q_0} \big)\dx \to \dashint_{Q_i} g^{Q_0}\dx
\end{align*}
as $m \to \infty$.
Hence, for almost every $x \in Q_i$, it holds that
\begin{align*}
g^{Q_0} - g  &= \lim_{m \to \infty} \big( f_{j_m} - (f_{j_m})_{Q_0} - (f_{j_m} - (f_{j_m})_{Q_i}) \big) \\
&= \lim_{m \to \infty} \big( (f_{j_m})_{Q_i} - (f_{j_m})_{Q_0} \big) = (g^{Q_0})_{Q_i} .
\end{align*}
This together with~\eqref{eq1} implies
\begin{align*}
\sum_{i=1}^\infty |Q_i|\biggl( \dashint_{Q_i} |g^{Q_0} - f_j - (g^{Q_0} - f_j)_{Q_i}| \dx \biggr)^p
&= \sum_{i=1}^\infty |Q_i| \biggl( \dashint_{Q_i} |g^{Q_0} - (g^{Q_0})_{Q_i} - g_j| \dx \biggr)^p \\
&= \sum_{i=1}^\infty|Q_i| \biggl( \dashint_{Q_i} |g - g_j| \dx\biggr)^p < \varepsilon^p ,
\end{align*}
whenever $j \geq j_\varepsilon$.
Since this holds for any collection $\{Q_i\}_{i\in\mathbb N}$, we can take the supremum over the collections to obtain
\begin{align*}
\lVert g^{Q_0} - f_{j}\rVert_{JN^d_{p}(Q_0)} < \varepsilon ,
\end{align*}
whenever $j \geq j_\varepsilon$.
This concludes that $g^{Q_0} = (g^{Q_0} - f_{j}) + f_{j} \in JN^d_{p}(Q_0)$ and $f_j$ converges to $g^{Q_0}$ in $JN^d_p(Q_0)$ as $j\to\infty$.

\end{proof}

\end{document}